\documentclass{article}
\usepackage{amsmath,amsthm,amsfonts,amssymb,amscd,euscript}

\theoremstyle{plain}
\newtheorem{thm}{Theorem}[section]

\newtheorem{lem}[thm]{Lemma}
\newtheorem{cor}[thm]{Corollary}

\theoremstyle{definition}
\newtheorem{defn}[thm]{Definition}

\numberwithin{equation}{subsection}

\theoremstyle{remark}
\newtheorem{claim}[thm]{Claim}

\newcommand{\cellsize}{10}
\newlength{\cellsz} 
\newsavebox{\cell}
\sbox{\cell}{\begin{picture}(\cellsize,\cellsize)
\put(0,0){\line(1,0){\cellsize}} \put(0,0){\line(0,1){\cellsize}}
\put(\cellsize,0){\line(0,1){\cellsize}}
\put(0,\cellsize){\line(1,0){\cellsize}}
\end{picture}}
\newcommand\cellify[1]{\def\thearg{#1}\def\nothing{}%
\ifx\thearg\nothing \vrule width0pt height\cellsz depth0pt\else
\hbox to 0pt{\usebox{\cell} \hss}\fi%
\vbox to \cellsz{ \vss \hbox to \cellsz{\hss$#1$\hss} \vss}}
\newcommand\tableau[1]{\vtop{\let\\\cr
\baselineskip -16000pt \lineskiplimit 16000pt \lineskip 0pt
\ialign{&\cellify{##}\cr#1\crcr}}}

\begin{document}

\begin{center}
{\large CHARACTERISTIC $2$ APPROACH TO BIVARIATE INTERPOLATION
PROBLEMS

\vspace{.5cm}
 Kyungyong Lee}

\vspace{.3cm} Department of Mathematics, University of Michigan, Ann
Arbor, Michigan, USA

kyungl@umich.edu
\end{center}






\begin{abstract}
We investigate bivariate interpolation problems in characteristic 2.
Given a nonnegative integer $t$, we describe all the sub-linear
systems generated by monomials, in which there is no curve passing
through a general point with multiplicity at least $2^t$. As an
application, we show that a certain linear system of plane curves
with 10 base points is non-special.
\end{abstract}

\section{Introduction}

We deal with bivariate interpolation problems in an infinite field
$\mathbb{K}$ of characteristic 2. Characteristic 2 condition is not
a very restrictive assumption because solvability of an
interpolation problem in characteristic 2 implies solvability of the
same problem in characteristic 0. Moreover working in characteristic
2 has many advantages. For instance, we do not need to take care of
signs when we compute the determinants of matrices.

For the reader of this paper, an acquaintance with the notions and
methods in \cite{BM:top}, \cite{LJR} would be very useful. We more
or less follow the notations of \cite{LL:biv}. Given a fixed set $S$
of lattice points $(i,j),i,j\geq 0$, the sub-linear system
$\mathcal{P}(S)$ with respect to $S$ consists of $$
P(x,y)=\sum_{(i,j) \in S} a_{i,j}x^i y^j \in \mathbb{K}[x,y].$$
Notice that unlike in \cite{LL:biv}, we do not necessarily assume
$S$ to be a lower set.

Throughout this note, the coordinates of lattice points are always
nonnegative. Let $T_m$ be the triangle of all $(i,j)$ with $i+j\leq
m-1$. $T_m$ contains $|T_m|=\frac{1}{2}m(m+1)$ lattice points.

For a set of $n$ distinct interpolation knots $Z=\{z_q:=(x_q,
y_q)\}_{q=1}^n$ in $\mathbb{K}^2$, it is interesting to study
(sub-)linear systems of plane curves passing through $Z$ with
multiplicity $\geq m_q$ at each point $z_q$ (for example, see
\cite{BM:top}, \cite{CM:degen}, \cite{CM:lin}, \cite{H:une},
\cite{LL:biv}, \cite{Na:hil}, \cite{Ha:nagata}, \cite{HR:nagata},
\cite{SS:nagata} and references therein). To put it in another way,
we are interested in solving the interpolation problem
$$
\frac{1}{\alpha ! \beta
!}\cdot{\frac{\partial^{\alpha+\beta}P}{\partial x^{\alpha}\partial
y^{\beta}}}\Big|_{z_q} = 0, \text{ }\text{ }\text{ }\text{ }\text{ }
(\alpha, \beta)\in T_{m_q}, \text{ }\text{ }\text{ }\text{ }\text{ }
q=1,\cdots, n. \text{ }\text{ }\text{ }\text{ }\text{ }\text{
}\text{ }\text{ }\text{ }\text{ }(*)
$$ Note that we do not necessarily require $|S|=\sum_{q=1}^n
|T_{m_q}|$. As in \cite{LL:biv}, we say that an interpolation scheme
is \emph{almost surely solvable} or \emph{almost regular} if $(*)$
is solvable for almost all $Z \in (\mathbb{K}^2)^n$. Since the right
hand sides in $(*)$ are $0$, our interpolation problem is almost
regular if and only if $(*)$ has only trivial solution for almost
all $Z$. We remark that if $S$ has a double element $(i,j)$ then a
nontrivial solution $P=x^i y^j +x^i y^j$ exists hence the
interpolation problem is never almost regular.

Since it is natural to ask which (sub-)linear systems are almost
regular, there has been some interest in trying to understand it.
But up to now, even in the case $n=|Z|=1$ there have been no
explicit criterions in positive characteristic, and no other
criterions in characteristic 0 than Bezout-Dumnicki lemma
\cite[Lemma 20]{Du:red} which gives a sufficient but not necessary
condition for a (sub-)linear system to be almost regular.

In this note, we completely solve the interpolation problem in
characteristic 2 in the case when $n=|Z|=1$ and $m=m_1=2^t$ $( t \in
\mathbb{N})$. In other words, given $t \in \mathbb{N}$, we describe
all the sub-linear systems generated by monomials, in which there is
no curve passing through a general point with multiplicity $\geq
2^t$. This case is already interesting in its own right, and is
indispensable for dealing with the cases of $n\geq 2$ knots (c.f.
\cite[Proposition 12]{Du:red}).

When $n=1$, our interpolation scheme $\langle S,T_m \rangle$ becomes
$$
\frac{1}{\alpha ! \beta
!}\cdot{\frac{\partial^{\alpha+\beta}P}{\partial x^{\alpha}\partial
y^{\beta}}}\Big|_{z_1} = 0, \text{ }\text{ }\text{ }\text{ }\text{ }
(\alpha, \beta)\in T_{m}\text{ }\text{ }\text{ }\text{ }\text{
}\text{ }\text{ }\text{ }\text{ }\text{ }(**).$$ Our main theorem
shows that the interpolation problem $\langle S, T_{2^{t}}\rangle$
is inductive on $t$, in other words, almost regularity of $\langle
S, T_{2^{t+1}} \rangle$ can be determined by some interpolation
problems with $T_{2^{t}}$.

\begin{thm}\label{main}
Let $\mathbb{K}$ be an infinite field of characteristic 2. The following statements are equivalent.\\
$(i)$ The interpolation problem $\langle S, T_{2^{t+1}} \rangle$ is almost regular.\\
$(ii)$ There is no triple $(U,V,W) \subset S|_{\emph{h,v}}$,
$S|_{\emph{v,d}}$, $S|_{\emph{d,h}}$ $\emph{(}$see Definition
~\ref{def_hvd}$\emph{)}$ of subsets such that
$$\aligned
\bullet\text{ each}&\text{ of }\sum_{(i,j) \in U} x^i y^j,
\sum_{(i,j) \in V} x^i y^j,\sum_{(i,j) \in W} x^i y^j \text{ is a
solution of }(**)\text{ for }m=2^t, \text{ } \text{ } \text{ }
\text{ } \text{ } \text{ } \text{ } \text{ } \text{ } \text{ }
\text{ } \text{ }\text{ } \text{ } \text{ } \text{ } \text{ } \text{
} \text{ } \text{ }
\text{ } \text{ } \text{ } \text{ } \text{ } \text{ } \\
\bullet\text{ at m}&\text{ost one of } U,V,W \text{is empty,}\text{
} \text{ } \text{ } \text{ } \text{ } \text{ } \text{ } \text{ }
\text{ } \text{ } \text{ } \text{ } \text{ } \text{ }\text{ } \text{
} \text{ } \text{ } \text{ } \text{ } \text{ } \text{ } \text{ }
\text{ } \text{ } \text{ } \text{ } \text{ }(\dag)\\
\bullet\text{ and } &(U-V) \cup (V-U)=W\\
&(V-W) \cup (W-V)=U\\
&(W-U) \cup (U-W)=V.
\endaligned$$
\end{thm}

The following two corollaries are often useful in practice.

\begin{cor}\label{cor1}
If the interpolation problem $\langle S, T_{2^{t+1}} \rangle$ is
almost regular, then the three of $\langle S|_{\emph{hori}},
T_{2^t}\rangle$, $\langle S|_{\emph{vert}}, T_{2^t} \rangle$, and
$\langle S|_{\emph{diag}}, T_{2^t}\rangle$ are all almost regular.
\end{cor}

\begin{cor}\label{cor2}
If at least two of the three interpolation problems $\langle
S|_{\emph{h,v}}, T_{2^t}\rangle$, $\langle S|_{\emph{v,d}}, T_{2^t}
\rangle$, and $\langle S|_{\emph{d,h}}, T_{2^t}\rangle$ are almost
regular, then the interpolation problem $\langle S, T_{2^{t+1}}
\rangle$ is almost regular.
\end{cor}

For example, if $S=\{(0,0),(0,1),$ $(1,0),(0,2),$ $(1,1),(2,0),$
$(1,2),(2,1),$ $(3,0),(1,3)\}$ and $m=2^2$, then $\langle S, T_{2^2}
\rangle$
\begin{tabular}{|r|r|r|r|}
\hline
0 & 1 & 0 & 0 \\
 \hline
1 & 1 & 0 & 0 \\
\hline
1 & 1 & 1 & 0 \\
\hline
1 & 1 & 1 & 1 \\
\hline
\end{tabular}
(see the preceding sentence of Definition ~\ref{def_hvd}) is almost
regular because $\langle S|_{\text{v,d}}, T_{2}\rangle$
\begin{tabular}{|r|r|}
\hline
0 & 1 \\
 \hline
1 & 1 \\
\hline
\end{tabular} and $\langle S|_{\text{d,h}}, T_{2} \rangle$
\begin{tabular}{|r|r|}
\hline
1 & 0 \\
 \hline
1 & 1 \\
\hline
\end{tabular}
are almost regular.

As an application of Theorem ~\ref{main}, we show that, without the
aid of a computer,  the linear system of plane curves of degree $26$
passing through $10$ general base points with $m_1=m_2=9$,
$m_3=...=m_{10}=8$ is empty. Our future project is to generalize
this to bigger multiplicity cases.

\bigskip

$\emph{Acknowledgements.}$ The author is much indebted to Ivan
Petrakiev, who read these notes carefully and gave helpful comments.
He also would like to thank Rob Lazarsfeld, Carl de Boor, Marcin
Dumnicki, Rick Miranda, Amos Ron, Shayne Waldron, Yuan Xu for
valuable correspondences.


\section{Definitions}

Throughout this note, for each lattice point $(i,j)$, the vector
whose $(a,b)$-th component $((a,b) \in T_{2^t})$ is ${i \choose a}{j
\choose b} (\text{mod } 2)$ will be denoted by $v^t_{i,j}$. This is
nothing but the vector consisting of the coefficients of a column in
an interpolation matrix(c.f. \cite{LL:biv} p.670). We always arrange
$(a,b)$-components with respect to the total degree order, that is,
$(0,0) < (0,1) < (1,0) < (0,2) < (1,1) < (2,0) < (0,3) < \cdots$.
For example,
$$v^0_{0,0}=\left(
\begin{array}{c}
                                      1 \\
                                       \end{array}
                                  \right),$$

$$
v^1_{0,0}= \left(
                                   \begin{array}{c}
                                      1 \\
                                      0 \\
                                      0 \\
                                    \end{array}
                                  \right),
                                  v^1_{0,1}= \left(
                                   \begin{array}{c}
                                      1 \\
                                      1 \\
                                      0 \\
                                    \end{array}
                                  \right),
                                  v^1_{1,0}= \left(
                                   \begin{array}{c}
                                      1 \\
                                      0 \\
                                      1 \\
                                    \end{array}
                                  \right),
                                  v^1_{1,1}= \left(
                                   \begin{array}{c}
                                      1 \\
                                      1 \\
                                      1 \\
                                    \end{array}
                                  \right),
                                  $$

$$
v^2_{0,0}= \left(
                                   \begin{array}{c}
                                      1 \\
                                      0 \\
0 \\
                                      0 \\
                                      0 \\
                                      0 \\
                                      0 \\
                                      0 \\
                                      0 \\
                                      0 \\
                                    \end{array}
                                  \right),
                                  v^2_{1,2}= \left(
                                   \begin{array}{c}
                                      1 \\
                                      0 \\
                                      1 \\
                                      1 \\
                                      0 \\
                                      0 \\
                                      0 \\
                                      1 \\
                                      0 \\
                                      0 \\
                                      \end{array}
                                  \right),
                                    v^2_{3,3}= \left(
                                   \begin{array}{c}
                                      1 \\
                                      1 \\
                                      1 \\
                                      1 \\
                                      1 \\
                                      1 \\
                                      1 \\
                                      1 \\
                                      1 \\
                                      1 \\
                                      \end{array}
                                  \right).
$$

\begin{defn}
We say that $S$ is $2^t$-\emph{independent} if $\{v^t_{i,j} | (i,j)
\in S\}$ are linearly independent. This is equivalent to saying that
$\langle S, T_{2^{t}} \rangle$ is almost regular.
\end{defn}

The problem of deciding whether $S$ is $2^t$-independent can be
reduced modulo $2^t$. Let $B_t := \{(x,y) | x,y=0,1, \cdots, 2^t-1
\}$ and consider the natural projection $\rho : S \longrightarrow
B_t$ defined by $(i,j)\in S  \longmapsto  (x,y)\in B_t$ where $x
\equiv i (\text{mod } 2^t), y \equiv j (\text{mod } 2^t)$.

\begin{lem}\label{mod_two_red}
$S$ is $2^t$-independent if and only if the image $\rho(S)$,
counting multiple elements, is $2^t$-independent.
\end{lem}
\begin{proof}
It follows from the fact that if $u \equiv v (\text{mod } 2^t)$ and
$0\leq z \leq 2^t-1$ then ${u \choose z} \equiv {v \choose z}
(\text{mod } 2).$
\end{proof}

Thanks to Lemma ~\ref{mod_two_red}, we can and will assume that $S$
is a subset (possibly counting multiple elements) of $B_t$. Since we
use induction on $t$ in the proof of Theorem ~\ref{main}, we will
sometimes write a pair $(B_t, S)$ in place of $S \subset B_t$ to
avoid confusion. The visualization of $(B_t, S)$ will be used
frequently, for instance,

\bigskip

\noindent$\text{ }\text{ }\text{ }\text{ }\text{ }\text{ }\text{
}\text{ }\text{ }\text{ }\text{ }\text{ }\text{ }\text{ }\text{
}\text{ }\text{ }\text{ }\text{ }\text{ }\text{ }\text{ }\text{
}\text{ }\text{ }\text{ }\text{ }\text{ }\text{ }\text{ }\text{
}\text{ }\text{ }\text{ }\text{ }\text{ }\text{ }\text{ }\text{
}\text{ }\text{ }\text{ }\text{ }\text{ }\text{ }\text{ }\text{
}\text{ }\text{ }\text{ }\text{ }\text{ }\text{ }\text{ }\text{
}\text{ }\text{ }\text{ }\text{ }\text{ }\text{ }\text{ }\text{
}\text{ }\text{ }\text{ }\text{ }\text{ }\text{ }\text{ }\text{
}\text{ }\text{ }\text{ }\text{ }\text{ }\text{ }\text{ }\text{
}\text{ }$
\begin{picture}(50,45) \put(10,10){\line(0,1){40}}
\put(10,10){\line(1,0){40}} \put(50,10){\line(0,1){40}}
\put(10,50){\line(1,0){40}} \put(20,10){\line(0,1){40}}
\put(30,10){\line(0,1){40}}
\put(40,10){\line(0,1){40}}\put(10,20){\line(1,0){40}}
\put(10,30){\line(1,0){40}} \put(10,40){\line(1,0){40}}

\put(12,41){0}\put(22,41){0}\put(32,41){0}\put(42,41){1}
\put(12,31){0}\put(22,31){2}\put(32,31){0}\put(42,31){0}
\put(12,21){0}\put(22,21){0}\put(32,21){0}\put(42,21){0}
\put(12,11){1}\put(22,11){0}\put(32,11){0}\put(42,11){0}

\put(12,03){\tiny{0}}\put(22,03){\tiny{1}}\put(32,03){\tiny{2}}\put(42,03){\tiny{3}}
\put(04,13){\tiny{0}}\put(04,23){\tiny{1}}\put(04,33){\tiny{2}}\put(04,43){\tiny{3}}

\put(-270,25){$S=\{(0,0),(1,2),(1,2),(3,3)\} \subset B_2$
corresponds to } \put(52,25){.}
\end{picture}

\begin{defn}\label{def_hvd}
Suppose that $S \subset B_{t+1}$ does not have multiple points. Then
we define $S|_{\text{h,v}} \subset B_t$ as follows. Given an element
$(i,j) \in B_t$, the element $(i,j)$ belongs to $S|_{\text{h,v}}$ if
and only if one of the four subsets
\begin{equation}\label{hvd_cond}
\aligned
& \{(i,j), (i+2^t,j)\},\\
& \{(i,j+2^t), (i+2^t,j+2^t)\},\\
& \{(i,j), (i,j+2^t)\},\\
& \{(i+2^t,j), (i+2^t,j+2^t)\}
\endaligned
\end{equation} is contained in $S$. If all the four are contained in $S$ then
we call $(i,j)$ a double element. We notice that \emph{h} stands for
horizontal, \emph{v} vertical, and \emph{d} diagonal. Illustrations
are given below.

$\text{ }\text{ }\text{ }\text{ }$
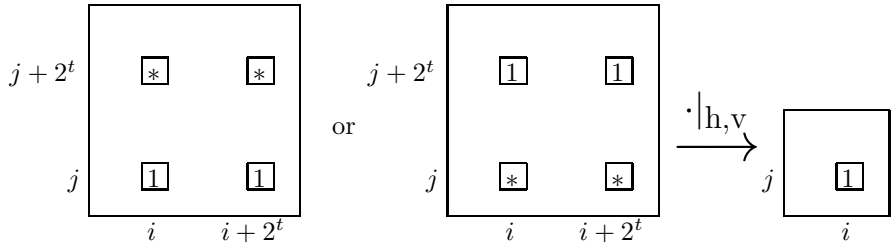
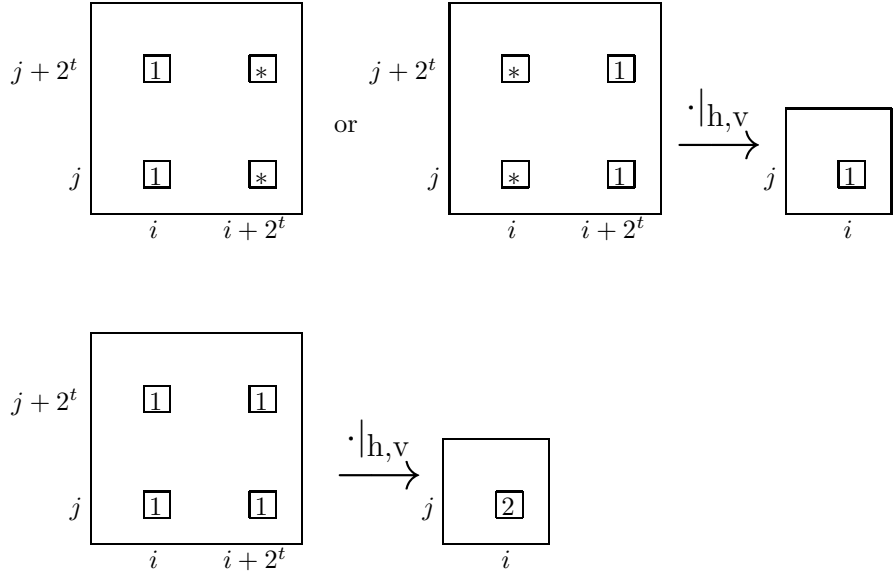
\begin{picture}(120,100)
\put(0,0){\line(0,1){80}} \put(0,0){\line(1,0){80}}
\put(80,0){\line(0,1){80}} \put(0,80){\line(1,0){80}}

\put(20,10){\line(1,0){\cellsize}}\put(20,20){\line(1,0){\cellsize}}
\put(20,10){\line(0,1){\cellsize}}\put(30,10){\line(0,1){\cellsize}}

\put(60,10){\line(1,0){\cellsize}}\put(60,20){\line(1,0){\cellsize}}
\put(60,10){\line(0,1){\cellsize}}\put(70,10){\line(0,1){\cellsize}}

\put(20,50){\line(1,0){\cellsize}}\put(20,60){\line(1,0){\cellsize}}
\put(20,50){\line(0,1){\cellsize}}\put(30,50){\line(0,1){\cellsize}}

\put(60,50){\line(1,0){\cellsize}}\put(60,60){\line(1,0){\cellsize}}
\put(60,50){\line(0,1){\cellsize}}\put(70,50){\line(0,1){\cellsize}}

\put(22,-9){$i$} \put(-8,11){$j$}\put(50,-9){$i+2^t$}
\put(-30,51){$j+2^t$} \put(22,11){1} \put(62,11){1} \put(22,51){$*$}
\put(62,51){$*$} \put(92,31){$\text{or }\text{ }\text{ }\text{
}\text{ }$}
\end{picture}  $\text{ }\text{ }\text{ }\text{ }$\begin{picture}(125,100)
\put(0,0){\line(0,1){80}} \put(0,0){\line(1,0){80}}
\put(80,0){\line(0,1){80}} \put(0,80){\line(1,0){80}}

\put(20,10){\line(1,0){\cellsize}}\put(20,20){\line(1,0){\cellsize}}
\put(20,10){\line(0,1){\cellsize}}\put(30,10){\line(0,1){\cellsize}}

\put(60,10){\line(1,0){\cellsize}}\put(60,20){\line(1,0){\cellsize}}
\put(60,10){\line(0,1){\cellsize}}\put(70,10){\line(0,1){\cellsize}}

\put(20,50){\line(1,0){\cellsize}}\put(20,60){\line(1,0){\cellsize}}
\put(20,50){\line(0,1){\cellsize}}\put(30,50){\line(0,1){\cellsize}}

\put(60,50){\line(1,0){\cellsize}}\put(60,60){\line(1,0){\cellsize}}
\put(60,50){\line(0,1){\cellsize}}\put(70,50){\line(0,1){\cellsize}}

\put(22,-9){$i$} \put(-8,11){$j$}\put(50,-9){$i+2^t$}
\put(-30,51){$j+2^t$} \put(22,11){$*$} \put(62,11){$*$}
\put(22,51){1} \put(62,51){1} \put(85,21){\huge$\overset{
\cdot|_{\text{h,v}}}\longrightarrow$}
\end{picture}
\begin{picture}(100,100)
\put(0,0){\line(0,1){\cellsize}} \put(0,10){\line(0,1){\cellsize}}
\put(0,20){\line(0,1){\cellsize}} \put(0,30){\line(0,1){\cellsize}}

\put(0,0){\line(1,0){\cellsize}} \put(10,0){\line(1,0){\cellsize}}
\put(20,0){\line(1,0){\cellsize}} \put(30,0){\line(1,0){\cellsize}}

\put(40,0){\line(0,1){\cellsize}} \put(40,10){\line(0,1){\cellsize}}
\put(40,20){\line(0,1){\cellsize}}
\put(40,30){\line(0,1){\cellsize}} \put(0,40){\line(1,0){\cellsize}}
\put(10,40){\line(1,0){\cellsize}}
\put(20,40){\line(1,0){\cellsize}}
\put(30,40){\line(1,0){\cellsize}}

\put(20,10){\line(1,0){\cellsize}}
\put(20,20){\line(1,0){\cellsize}}
\put(20,10){\line(0,1){\cellsize}}
\put(30,10){\line(0,1){\cellsize}}

\put(22,11){1} \put(22,-9){$i$} \put(-8,11){$j$}
\end{picture}

\bigskip
\bigskip

$\text{ }\text{ }\text{ }\text{ }$\begin{picture}(120,100)
\put(0,0){\line(0,1){80}} \put(0,0){\line(1,0){80}}
\put(80,0){\line(0,1){80}} \put(0,80){\line(1,0){80}}

\put(20,10){\line(1,0){\cellsize}}\put(20,20){\line(1,0){\cellsize}}
\put(20,10){\line(0,1){\cellsize}}\put(30,10){\line(0,1){\cellsize}}

\put(60,10){\line(1,0){\cellsize}}\put(60,20){\line(1,0){\cellsize}}
\put(60,10){\line(0,1){\cellsize}}\put(70,10){\line(0,1){\cellsize}}

\put(20,50){\line(1,0){\cellsize}}\put(20,60){\line(1,0){\cellsize}}
\put(20,50){\line(0,1){\cellsize}}\put(30,50){\line(0,1){\cellsize}}

\put(60,50){\line(1,0){\cellsize}}\put(60,60){\line(1,0){\cellsize}}
\put(60,50){\line(0,1){\cellsize}}\put(70,50){\line(0,1){\cellsize}}

\put(22,-9){$i$} \put(-8,11){$j$}\put(50,-9){$i+2^t$}
\put(-30,51){$j+2^t$} \put(22,11){1} \put(62,11){$*$} \put(22,51){1}
\put(62,51){$*$} \put(92,31){$\text{or }\text{ }\text{ }\text{
}\text{ }$}
\end{picture}  $\text{ }\text{ }\text{ }\text{ }$\begin{picture}(125,100)
\put(0,0){\line(0,1){80}} \put(0,0){\line(1,0){80}}
\put(80,0){\line(0,1){80}} \put(0,80){\line(1,0){80}}

\put(20,10){\line(1,0){\cellsize}}\put(20,20){\line(1,0){\cellsize}}
\put(20,10){\line(0,1){\cellsize}}\put(30,10){\line(0,1){\cellsize}}

\put(60,10){\line(1,0){\cellsize}}\put(60,20){\line(1,0){\cellsize}}
\put(60,10){\line(0,1){\cellsize}}\put(70,10){\line(0,1){\cellsize}}

\put(20,50){\line(1,0){\cellsize}}\put(20,60){\line(1,0){\cellsize}}
\put(20,50){\line(0,1){\cellsize}}\put(30,50){\line(0,1){\cellsize}}

\put(60,50){\line(1,0){\cellsize}}\put(60,60){\line(1,0){\cellsize}}
\put(60,50){\line(0,1){\cellsize}}\put(70,50){\line(0,1){\cellsize}}

\put(22,-9){$i$} \put(-8,11){$j$}\put(50,-9){$i+2^t$}
\put(-30,51){$j+2^t$} \put(22,11){$*$} \put(62,11){1}
\put(22,51){$*$} \put(62,51){1} \put(85,21){\huge$\overset{
\cdot|_{\text{h,v}}}\longrightarrow$}
\end{picture}
\begin{picture}(100,100)
\put(0,0){\line(0,1){\cellsize}} \put(0,10){\line(0,1){\cellsize}}
\put(0,20){\line(0,1){\cellsize}} \put(0,30){\line(0,1){\cellsize}}

\put(0,0){\line(1,0){\cellsize}} \put(10,0){\line(1,0){\cellsize}}
\put(20,0){\line(1,0){\cellsize}} \put(30,0){\line(1,0){\cellsize}}

\put(40,0){\line(0,1){\cellsize}} \put(40,10){\line(0,1){\cellsize}}
\put(40,20){\line(0,1){\cellsize}}
\put(40,30){\line(0,1){\cellsize}} \put(0,40){\line(1,0){\cellsize}}
\put(10,40){\line(1,0){\cellsize}}
\put(20,40){\line(1,0){\cellsize}}
\put(30,40){\line(1,0){\cellsize}}

\put(20,10){\line(1,0){\cellsize}}
\put(20,20){\line(1,0){\cellsize}}
\put(20,10){\line(0,1){\cellsize}}
\put(30,10){\line(0,1){\cellsize}}

\put(22,11){1} \put(22,-9){$i$} \put(-8,11){$j$}
\end{picture}

\bigskip
\bigskip

$\text{ }\text{ }\text{ }\text{ }$\begin{picture}(130,100)
\put(0,0){\line(0,1){80}} \put(0,0){\line(1,0){80}}
\put(80,0){\line(0,1){80}} \put(0,80){\line(1,0){80}}

\put(20,10){\line(1,0){\cellsize}}\put(20,20){\line(1,0){\cellsize}}
\put(20,10){\line(0,1){\cellsize}}\put(30,10){\line(0,1){\cellsize}}

\put(60,10){\line(1,0){\cellsize}}\put(60,20){\line(1,0){\cellsize}}
\put(60,10){\line(0,1){\cellsize}}\put(70,10){\line(0,1){\cellsize}}

\put(20,50){\line(1,0){\cellsize}}\put(20,60){\line(1,0){\cellsize}}
\put(20,50){\line(0,1){\cellsize}}\put(30,50){\line(0,1){\cellsize}}

\put(60,50){\line(1,0){\cellsize}}\put(60,60){\line(1,0){\cellsize}}
\put(60,50){\line(0,1){\cellsize}}\put(70,50){\line(0,1){\cellsize}}

\put(22,-9){$i$} \put(-8,11){$j$} \put(50,-9){$i+2^t$}
\put(-30,51){$j+2^t$} \put(22,11){$1$} \put(62,11){$1$}
\put(22,51){1} \put(62,51){1} \put(92,21){\huge$\overset{
\cdot|_{\text{h,v}}}\longrightarrow$}
\end{picture}
\begin{picture}(100,100)
\put(0,0){\line(0,1){\cellsize}} \put(0,10){\line(0,1){\cellsize}}
\put(0,20){\line(0,1){\cellsize}} \put(0,30){\line(0,1){\cellsize}}

\put(0,0){\line(1,0){\cellsize}} \put(10,0){\line(1,0){\cellsize}}
\put(20,0){\line(1,0){\cellsize}} \put(30,0){\line(1,0){\cellsize}}

\put(40,0){\line(0,1){\cellsize}} \put(40,10){\line(0,1){\cellsize}}
\put(40,20){\line(0,1){\cellsize}}\put(40,30){\line(0,1){\cellsize}}
\put(0,40){\line(1,0){\cellsize}} \put(10,40){\line(1,0){\cellsize}}
\put(20,40){\line(1,0){\cellsize}}\put(30,40){\line(1,0){\cellsize}}

\put(20,10){\line(1,0){\cellsize}}
\put(20,20){\line(1,0){\cellsize}}
\put(20,10){\line(0,1){\cellsize}}
\put(30,10){\line(0,1){\cellsize}}

\put(22,11){2} \put(22,-9){$i$} \put(-8,11){$j$}
\end{picture}

\bigskip
\bigskip

$S|_{\text{v,d}} \subset B_t$ (resp. $S|_{\text{d,h}} \subset B_t$)
can be similarly defined. In the above definition we replace
$(~\ref{hvd_cond})$ by
$$
\aligned
& \{(i,j), (i,j+2^t)\},\{(i+2^t,j), (i+2^t,j+2^t)\},\\
&\{(i,j), (i+2^t,j+2^t)\},\{(i,j+2^t),(i+2^t,j)\} \\
&\text{ }\text{ }( \text{ }\text{
}\text{resp. } \{(i,j), (i+2^t,j+2^t)\},\{(i,j+2^t),(i+2^t,j)\},\\
& \text{ }\text{ }\text{ }\text{ }\text{ }\text{ }\text{ }\text{
}\text{ }\text{ }\text{ }\text{ }\text{ }\{(i,j), (i+2^t,j)\},
\{(i,j+2^t), (i+2^t,j+2^t)\}\text{ }\text{ }).
\endaligned
$$

To define $ S|_{\text{hori}}$, $S|_{\text{vert}}$,
$S|_{\text{diag}}$ respectively, we replace $(~\ref{hvd_cond})$ by
$$ \aligned
&\big( \{(i,j), (i+2^t,j)\}, \{(i,j+2^t), (i+2^t,j+2^t)\} \big),\\
&\big( \{(i,j), (i,j+2^t)\},\{(i+2^t,j), (i+2^t,j+2^t)\} \big),\\
&\big( \{(i,j), (i+2^t,j+2^t)\},\{(i,j+2^t),(i+2^t,j)\} \big)
\endaligned
$$respectively.

\end{defn}


             \section{Proofs}

\begin{proof}[Proof of Theorem ~\ref{main}]
$\text{ }$

\smallskip

\noindent$(i) \Longrightarrow  (ii): $

Before giving a proof, we introduce some notations. Let $V^t$ be the
${{2^{t}+1}\choose {2}}$-dimensional vector space over
$\mathbb{F}_2$ consisting of $(a,b)$-components $((a,b) \in
T_{2^{t}})$, so that $v^t_{i,j} \in V^{t}$. We decompose $V^{t+1}$
into $V^t_w \oplus V^t_y \oplus V^t_z$, where $V^t_w$ consists of
$(a,b)$-components with $(a,b) \in B_t$, $V^t_y$ with $b\geq 2^t$,
and $V^t_z$ with $a \geq 2^t$. As a matter of fact, we consider the
following isomorphism of vector spaces:
$$
\tau : V^{t+1} \simeq V^t_w \oplus V^t_y \oplus V^t_z$$

$$
\left(
                                  \begin{array}{c}
(0,0)\text{-th component of }v\\
(0,1)\text{-th component of }v\\
(1,0)\text{-th component of }v\\
(0,2)\text{-th component of }v\\
(1,1)\text{-th component of }v\\
(2,0)\text{-th component of }v\\
\vdots\\
\vdots\\
\vdots\\
\vdots\\
\vdots\\
\vdots\\
\vdots\\
\vdots\\
(2^{t+1}-3,2)\text{-th}\\
(2^{t+1}-2,1)\text{-th}\\
(2^{t+1}-1,0)\text{-th}\\
\end{array}
\right)=v \overset{\tau}\mapsto \begin{array}{c}
w\\
\oplus\\
y\\
\oplus\\
z\\
\end{array}=\begin{array}{c}
\left(
                                  \begin{array}{c}
(0,0)\text{-th}\\
(0,1)\text{-th}\\
(1,0)\text{-th}\\
\vdots\\
(2^t-2,2^t-1)\text{-th}\\
(2^t-1,2^t-2)\text{-th}\\
(2^t-1,2^t-1)\text{-th}\\
                                     \end{array}
                                 \right)\\
\oplus\\
\left(
                                  \begin{array}{c}
(0,2^t)\text{-th}\\
(0,2^t+1)\text{-th}\\
(1,2^t)\text{-th}\\
\vdots\\
(2^{t}-1,2^t)\text{-th}\\
                                     \end{array}
                                 \right)\\
                                 \oplus\\
\left(
                                  \begin{array}{c}
(2^t,0)\text{-th}\\
(2^t,1)\text{-th}\\
(2^t+1,0)\text{-th}\\
\vdots\\
(2^{t+1}-1,0)\text{-th}\\
                                     \end{array}
                                 \right),
                                 \end{array}
$$ where we have only rearranged the order of components. Note that
$\text{dim}(V^t_w) = (2^t)^2$, $\text{dim}(V^t_y)=\text{dim}(V^t_z)$
$={{2^{t}+1}\choose {2}}$ and that $V^{t}_y$, $V^{t}_z$, and $V^t$
are isomorphic as vector spaces.

\bigskip

Seeking contradiction, suppose that there is a triple $(U,V,W)
\subset S|_{\text{h,v}}$, $S|_{\text{v,d}}$, $S|_{\text{d,h}}$ of
subsets satisfying the three conditions in $(\dag)$ in the statement
of Theorem ~\ref{main}. Without loss of generality, we assume that
$U$ and $V$ are nonempty.

 By using the proof of Lemma ~\ref{mod_two_red},
we observe that

\medskip

\begin{picture}(30,35)
\put(25,19){$(a,b)$-th component of} \put(0,4){$v^{t+1}_{i,j}+
v^{t+1}_{i+2^t,j}\text{ }( \text{or } v^{t+1}_{i,j+2^t}+
v^{t+1}_{i+2^t,j+2^t})$}

\put(210,21){0} \put(280,21){if $a < 2^t,$} \put(210,01){${i \choose
{a-2^t}} {j \choose b}$} \put(280,01){if $a \geq 2^t$,}

\put(185,11){=\text{ }\Big{\{}}
\end{picture}

\medskip

and

\medskip

\begin{picture}(30,35)
\put(25,19){$(a,b)$-th component of} \put(0,4){$v^{t+1}_{i,j}+
v^{t+1}_{i,j+2^t}\text{ }( \text{or } v^{t+1}_{i+2^t,j}+
v^{t+1}_{i+2^t,j+2^t})$}

\put(210,21){0} \put(280,21){if $b < 2^t,$} \put(210,01){${i \choose
{a}} {j \choose {b-2^t}}$} \put(280,01){if $b \geq 2^t$.}

\put(185,11){=\text{ }\Big{\{}}
\end{picture}

\medskip

\noindent Then \begin{equation}\label{vecsumh}\tau(v^{t+1}_{i,j}+
v^{t+1}_{i+2^t,j}) \text{ }( \text{or } v^{t+1}_{i,j+2^t}+
v^{t+1}_{i+2^t,j+2^t})=\vec{0} \oplus \vec{0} \oplus v^t_{i,j}\in
V^t_w \oplus V^t_y \oplus V^t_z,\end{equation} and
\begin{equation}\label{vecsumv}\tau(v^{t+1}_{i,j}+ v^{t+1}_{i,j+2^t})\text{ }( \text{or }
v^{t+1}_{i+2^t,j}+ v^{t+1}_{i+2^t,j+2^t}) = \vec{0} \oplus v^t_{i,j}
\oplus \vec{0} \in V^t_w \oplus V^t_y \oplus V^t_z.\end{equation}

\begin{lem}\label{lastvecsum}
Suppose that there is a triple $(U,V,W) \subset S|_{\text{h,v}}$,
$S|_{\text{v,d}}$, $S|_{\text{d,h}}$ of subsets satisfying the three
conditions in $(\dag)$ in the statement of Theorem ~\ref{main}. Then
we have
$$\aligned&\sum_{(i,j)\in U-V} \tau(v^{t+1}_{i,j}+
v^{t+1}_{i+2^t,j})\text{ } ( \emph{or }
v^{t+1}_{i,j+2^t}+ v^{t+1}_{i+2^t,j+2^t})\\
&\text{ }\text{ }\text{ }\text{ }\text{ }+ \sum_{(i,j)\in U\cap V}
\tau(v^{t+1}_{i,j}+ v^{t+1}_{i,j+2^t})\text{ }( \emph{or }
v^{t+1}_{i+2^t,j}+
v^{t+1}_{i+2^t,j+2^t}) \\
&\text{ }\text{ }\text{ }\text{ }\text{ }+ \sum_{(i,j)\in V-U}
\tau(v^{t+1}_{i,j}+ v^{t+1}_{i+2^t,j+2^t})\text{ } ( \emph{or }
v^{t+1}_{i,j+2^t}+ v^{t+1}_{i+2^t,j})\\
& = \vec{0}.\endaligned $$
\end{lem}
\begin{proof}[Proof of Lemma ~\ref{lastvecsum}]
By $(~\ref{vecsumh})$ and $(~\ref{vecsumv})$, we have
\begin{equation}\label{vecsumhv}
\aligned&\sum_{(i,j)\in U-V} \tau(v^{t+1}_{i,j}+
v^{t+1}_{i+2^t,j})\text{ } ( \text{or }
v^{t+1}_{i,j+2^t}+ v^{t+1}_{i+2^t,j+2^t})\\
&\text{ }\text{ }\text{ }\text{ }\text{ }+ \sum_{(i,j)\in U\cap V}
\tau(v^{t+1}_{i,j}+ v^{t+1}_{i,j+2^t})\text{ }( \text{or }
v^{t+1}_{i+2^t,j}+
v^{t+1}_{i+2^t,j+2^t}) \\
& = \sum_{(i,j)\in U-V} \vec{0} \oplus \vec{0} \oplus
v^t_{i,j}\\
&\text{ }\text{ }\text{ }\text{ }\text{ } + \sum_{(i,j)\in U\cap V}
\vec{0} \oplus v^t_{i,j} \oplus \vec{0}.\endaligned
\end{equation}

Applying the same argument as in $(~\ref{vecsumh})$ or
$(~\ref{vecsumv})$ for $(B_t, S|_{\text{v,d}})$, we get
\begin{equation}\label{vecsumvd}
\aligned&\sum_{(i,j)\in V-U} \tau(v^{t+1}_{i,j}+
v^{t+1}_{i+2^t,j+2^t})\text{ } ( \text{or }
v^{t+1}_{i,j+2^t}+ v^{t+1}_{i+2^t,j})\\
& = \sum_{(i,j)\in V-U} \vec{0} \oplus v^t_{i,j} \oplus
v^t_{i,j}.\endaligned
\end{equation}

On the other hand, since $(B_t, U)\subset (B_t, S|_{\text{h,v}})$,
$(B_t, V)\subset (B_t, S|_{\text{v,d}})$ and $(B_t, W)\subset (B_t,
S|_{\text{d,h}})$ satisfy the first condition in $(\dag)$, we have
$\sum_{(i,j)\in U} v^t_{i,j}=\vec{0}$, $\sum_{(i,j)\in V}
v^t_{i,j}=\vec{0}$ and $\sum_{(i,j)\in W} v^t_{i,j}=\vec{0}$.

 Adding $(~\ref{vecsumhv})$
and $(~\ref{vecsumvd})$ together gives
$$\aligned&\sum_{(i,j)\in U-V} \tau(v^{t+1}_{i,j}+
v^{t+1}_{i+2^t,j})\text{ } ( \text{or }
v^{t+1}_{i,j+2^t}+ v^{t+1}_{i+2^t,j+2^t})\\
&\text{ }\text{ }\text{ }\text{ }\text{ }+ \sum_{(i,j)\in U\cap V}
\tau(v^{t+1}_{i,j}+ v^{t+1}_{i,j+2^t})\text{ }( \text{or }
v^{t+1}_{i+2^t,j}+
v^{t+1}_{i+2^t,j+2^t}) \\
&\text{ }\text{ }\text{ }\text{ }\text{ }+ \sum_{(i,j)\in V-U}
\tau(v^{t+1}_{i,j}+ v^{t+1}_{i+2^t,j+2^t})\text{ } ( \text{or }
v^{t+1}_{i,j+2^t}+ v^{t+1}_{i+2^t,j})\\
& = \sum_{(i,j)\in U-V} \vec{0} \oplus \vec{0} \oplus v^t_{i,j}+
\sum_{(i,j)\in U\cap V} \vec{0} \oplus v^t_{i,j} \oplus \vec{0} +
\sum_{(i,j)\in V-U}
\vec{0} \oplus v^t_{i,j} \oplus v^t_{i,j}\\
& = \sum_{(i,j)\in V} \vec{0} \oplus v^t_{i,j} \oplus \vec{0} +
\sum_{(i,j)\in
(U-V)\cup(V-U)} \vec{0} \oplus \vec{0} \oplus v^t_{i,j}\\
& = \sum_{(i,j)\in V} \vec{0} \oplus v^t_{i,j} \oplus \vec{0} + \sum_{(i,j)\in W} \vec{0} \oplus \vec{0} \oplus v^t_{i,j}\\
& = \vec{0}.\endaligned $$ \end{proof}

We observe that $U-V\subset U\subset S|_{\text{h,v}}$ and $U-V
\subset W \subset S|_{\text{d,h}}$ imply $U-V \subset
S|_{\text{hori}}$. So, by Definition ~\ref{def_hvd},  if $(i,j)\in U
- V$ then either $\{(i,j),(i+2^t,j)\}$ or $\{(i,j+2^t), (i+2^t,
j+2^t)\}$ is contained in $S$. In the same manner, $V-U \subset
S|_{\text{diag}}$. If $(i,j)\in V-U$ then either $\{(i,j),
(i+2^t,j+2^t)\}$ or $\{(i,j+2^t), (i+2^t,j)\}$  is contained in $S$.
It is obvious that if $(i,j)\in U\cap V\subset S|_{\text{h,v}}\cap
S|_{\text{v,d}}$ then either $\{(i,j), (i,j+2^t)\}$ or $\{(i+2^t,j),
(i+2^t,j+2^t)\}$ is contained in $S$.

So Lemma ~\ref{lastvecsum} implies that there is a nonempty subset
$S'$ of $S$ such that $\sum_{(i,j)\in S'} v^{t+1}_{i,j}=\vec{0}$.
Therefore $(B_{t+1}, S)$ is not $2^{t+1}$-independent.

\bigskip
\bigskip
\bigskip
\medskip

\noindent$(i) \Longleftarrow  (ii): $

Suppose that $(B_{t+1}, S)$ is not $2^{t+1}$-independent. Then there
is  a nonempty minimal subset $S'$ of $S$ such that $\sum_{(i,j)\in
S'} v^{t+1}_{i,j}=\vec{0}$.

\begin{lem}\label{even}
$|S' \cap \{(i,j), (i,j+2^t), (i+2^t,j), (i+2^t,j+2^t)\}| \equiv 0
\text{ } (\emph{mod } 2)$ for any $(i,j) \in B_t$.
\end{lem}

To prove this lemma, we need the following claims. Let $w_{i,j} =
\tau_1 (v^{t+1}_{i,j})$ where $\tau_1$ is the natural projection
$V^{t+1}=V^t_w \oplus V^t_y \oplus V^t_z \longrightarrow V^t_w$.

\begin{claim}\label{lin_ind}
$w_{i,j}$ $( (i,j)\in B_t)$ are linearly independent. \end{claim}
\begin{proof}
It is easy to see that the matrix in which columns $w_{i,j}$ are
arranged with respect to the total degree order is an upper triangle
matrix with diagonal $(1,1, ... ,1)$.
\end{proof}

\begin{claim}\label{identical}
For any $(i,j) \in B_t$, the vectors $w_{i,j}, w_{i,j+2^t},
w_{i+2^t,j}, w_{i+2^t,j+2^t} $ are identical. \end{claim}
\begin{proof}
This is essentially the same as the proof of Lemma
~\ref{mod_two_red}.
\end{proof}

\begin{proof}[Proof of Lemma ~\ref{even}] Since the vector $\sum_{(i,j)\in S'}
v^{t+1}_{i,j}$ is the zero vector, its image under $\tau_1: V^{t+1}
\longrightarrow V^t_w$ is also zero vector. So we have
$$\aligned
\vec{0}&=\sum_{(i,j)\in S'} w_{i,j}\\
&=\sum_{(i,j)\in B_t}  |S' \cap \{(i,j), (i,j+2^t), (i+2^t,j),
(i+2^t,j+2^t)\}|  w_{i,j},
\endaligned
$$
where we have used Claim ~\ref{identical}. Then, by Claim
~\ref{lin_ind}, $|S' \cap \{(i,j), (i,j+2^t), (i+2^t,j),
(i+2^t,j+2^t)\}| \equiv 0\text{ } (\text{mod } 2)$ for every $(i,j)
\in B_t$. This completes the proof of Lemma ~\ref{even}.
\end{proof}

If $|S' \cap \{(i,j), (i,j+2^t), (i+2^t,j), (i+2^t,j+2^t)\}| =4$ for
some $(i,j) \in B_t$, then $ \{(i,j), (i,j+2^t), (i+2^t,j),
(i+2^t,j+2^t)\} \subset S' \subset S$. Then by definition
~\ref{def_hvd},  $(B_t, S|_{\text{h,v}})$ $(S|_{\text{v,d}},
S|_{\text{d,h}}$ also) has a double element hence is not
$2^t$-independent.

So we assume that $|S' \cap \{(i,j), (i,j+2^t), (i+2^t,j),
(i+2^t,j+2^t)\}| =0$ or $2$ for any $(i,j) \in B_t$. Of course we
consider only $(i,j) \in B_t$ with $|S' \cap \{(i,j), (i,j+2^t),
(i+2^t,j), (i+2^t,j+2^t)\}| =2$. For each of such $(i,j)$, depending
on elements in $S' \cap \{(i,j), (i,j+2^t), (i+2^t,j),
(i+2^t,j+2^t)\}$, we consider one and only one of the three vectors
\begin{equation}\label{add_two}
\aligned &\tau(v^{t+1}_{i,j}+ v^{t+1}_{i+2^t,j})\text{ }( \text{or }
v^{t+1}_{i,j+2^t}+ v^{t+1}_{i+2^t,j+2^t})=\vec{0}\oplus
\vec{0}\oplus v^t_{i,j} \in V^{t}_w \oplus V^{t}_y \oplus
V^{t}_z,\\
&\tau(v^{t+1}_{i,j}+ v^{t+1}_{i,j+2^t})\text{ }( \text{or }
v^{t+1}_{i+2^t,j}+ v^{t+1}_{i+2^t,j+2^t})=\vec{0}\oplus v^t_{i,j}
\oplus \vec{0} \in V^{t}_w \oplus V^{t}_y \oplus V^{t}_z,
\endaligned
\end{equation} or $$ \text{ }\tau(v^{t+1}_{i,j}+
v^{t+1}_{i+2^t,j+2^t})\text{ }( \text{or } v^{t+1}_{i+2^t,j}+
v^{t+1}_{i,j+2^t})=\vec{0}\oplus v^t_{i,j}\oplus v^t_{i,j} \in
V^{t}_w \oplus V^{t}_y \oplus V^{t}_z.
$$

These give rise to three disjoint (possibly empty but not all empty)
subsets $S_1, S_2, S_3$ of $B_t$ such that $$ \sum_{(i,j)\in S_1}
(\vec{0}\oplus \vec{0}\oplus v^t_{i,j}) + \sum_{(i,j)\in S_2}
(\vec{0}\oplus v^t_{i,j} \oplus \vec{0})  + \sum_{(i,j)\in S_3}
(\vec{0}\oplus v^t_{i,j}\oplus v^t_{i,j}) =\vec{0} $$ $$ \text{
}\text{ }\text{ }\text{ }\text{ }\text{ }\text{ }\text{ }\text{
}\text{ }\text{ }\text{ }\text{ }\text{ }\text{ }\text{ }\text{
}\text{ }\text{ }\text{ }\text{ }\text{ }\text{ }\text{ }\text{
}\text{ }\text{ }\text{ }\text{ }\text{ }\text{ }\text{ }\text{
}\text{ }\text{ }\text{ }\text{ }\text{ }\text{ }\text{ }\text{
}\text{ }\text{ }\text{ }\text{ }\text{ }\text{ }\text{ }\text{
}\text{ }\text{ }\text{ }\text{ }\text{ }\text{ }\text{ }\text{
}\text{ }\text{ }\text{ }\text{ }\text{ }\text{ }\text{ }\text{
}\text{ }\text{ }\text{ }\text{ }\text{ }\in V^{t}_w \oplus V^{t}_y
\oplus V^{t}_z.
$$
We show that the triple $(S_1\cup S_3$, $S_2\cup S_3$, $S_1\cup
S_2)$ satisfies the three conditions in $(\dag)$ in the statement of
Theorem ~\ref{main}. Recall that $V^{t}_y\cong V^{t}_z\cong V^t$.
This implies that $$\aligned &\sum_{(i,j)\in
S_1}v^t_{i,j}  + \sum_{(i,j)\in S_3}v^t_{i,j} =\vec{0} \in V^{t}_z\cong V^t,\\
&\sum_{(i,j)\in S_2}v^t_{i,j}  + \sum_{(i,j)\in S_3}v^t_{i,j}
=\vec{0} \in V^{t}_y \cong V^t,
\endaligned$$
hence
$$\aligned &\sum_{(i,j)\in S_1}v^t_{i,j}  + \sum_{(i,j)\in S_2}v^t_{i,j}+2\sum_{(i,j)\in
S_3}v^t_{i,j}\\
&=\sum_{(i,j)\in S_1}v^t_{i,j}  + \sum_{(i,j)\in S_2}v^t_{i,j}
=\vec{0} \in V^{t}.
\endaligned$$
Thus the first condition follows.

Since at least one of $S_1, S_2, S_3$ is nonempty, at most one of
$(S_1\cup S_3)$, $(S_2\cup S_3)$, and $(S_1\cup S_2)$ is empty. The
last condition follows from $S_1,S_2, S_3$ being disjoint. Proof of
Theorem ~\ref{main} is completed.
\end{proof}

\begin{proof}[Proof of Corollary ~\ref{cor1}]
If one, say $\langle S|_{\text{hori}}, T_{2^t}\rangle$, of the three
is not almost regular, then there is a nontrivial solution of
$(**)$. Let $\sum_{(i,j) \in U} x^i y^j$ be one of the solutions. If
$V=\emptyset$ and $W=U$ then the triple $(U,V,W)$ satisfies the
three conditions in $(\dag)$ in the statement of Theorem
~\ref{main}. So $\langle S, T_{2^{t+1}} \rangle$ is not almost
regular.
\end{proof}

\begin{proof}[Proof of Corollary ~\ref{cor2}]
Without loss of generality, suppose that $\langle S|_{\emph{h,v}},
T_{2^t}\rangle$ and $\langle S|_{\emph{v,d}}, T_{2^t}\rangle$ are
almost regular. If $\langle S, T_{2^{t+1}} \rangle$ were not almost
regular, then there would be a triple $(U,V,W)$ satisfying the three
conditions in $(\dag)$. But since $U$ and $V$ are empty, $(U,V,W)$
cannot satisfy the third condition.
\end{proof}

\section{Application}

We show that the linear system of plane curves of degree $d=26$
passing through $10$ general base points with $m_1=m_2=9$,
$m_3=...=m_{10}=8$ is empty. In this case, $S=T_{d+1}=T_{27}$. We
apply Dumnicki-Jarnicki reduction method \cite{DJ:new} and
Dumnicki's cutting diagram method \cite{Du:red} to $8$ points $z_1,
... , z_8$. Then we use Theorem ~\ref{main} several times to check
that the division below determines a unique nonzero (in
characteristic 2) monomial of the form $\prod_{q=1}^{10}
\frac{x_q^{\sum_{q \text{ is at } (i,j)} i} y_q^{\sum_{q \text{ is
at }(i,j)} j}}{x_q^{\sum_{(i,j)\in T_{m_q}} i} y_q^{\sum_{(i,j)\in
T_{m_q}} j}}$ in the determinant of the interpolation matrix $(*)$.
We remark that Bezout-Dumnicki lemma (\cite{Du:red}, Lemma 20) does
not work in positive characteristic and that even in characteristic
zero, it cannot cover the central region corresponding to the point
$z_{10}$.


$$\aligned
&2\\
&2\text{ }\text{ }2\\
&2\text{ }\text{ }2\text{ }\text{ }2\\
&2\text{ }\text{ }2\text{ }\text{ }2\text{ }\text{ }2\\
&2\text{ }\text{ }2\text{ }\text{ }2\text{ }\text{ }2\text{ }\text{ }2\\
&2\text{ }\text{ }2\text{ }\text{ }2\text{ }\text{ }2\text{ }\text{ }2\text{ }\text{ }2\\
&2\text{ }\text{ }2\text{ }\text{ }2\text{ }\text{ }2\text{ }\text{ }2\text{ }\text{ }2\text{ }\text{ }2\\
&2\text{ }\text{ }2\text{ }\text{ }2\text{ }\text{ }2\text{ }\text{ }2\text{ }\text{ }2\text{ }\text{ }2\text{ }\text{ }2\\
&2\text{ }\text{ }2\text{ }\text{ }2\text{ }\text{ }2\text{ }\text{ }2\text{ }\text{ }2\text{ }\text{ }2\text{ }\text{ }2\text{ }\text{ }2\\
&5\text{ }\text{ }5\text{ }\text{ }5\text{ }\text{ }5\text{ }\text{ }5\text{ }\text{ }5\text{ }\text{ }5\text{ }\text{ }5\text{ }\text{ }7\text{ }\text{ }7\\
&5\text{ }\text{ }5\text{ }\text{ }5\text{ }\text{ }5\text{ }\text{ }5\text{ }\text{ }5\text{ }\text{ }5\text{ }\text{ }7\text{ }\text{ }7\text{ }\text{ }7\text{ }\text{ }7\\
&5\text{ }\text{ }5\text{ }\text{ }5\text{ }\text{ }5\text{ }\text{ }5\text{ }\text{ }5\text{ }\text{ }7\text{ }\text{ }7\text{ }\text{ }7\text{ }\text{ }7\text{ }\text{ }7\text{ }\text{ }6\\
&5\text{ }\text{ }5\text{ }\text{ }5\text{ }\text{ }5\text{ }\text{ }5\text{ }\text{ }7\text{ }\text{ }7\text{ }\text{ }7\text{ }\text{ }7\text{ }\text{ }7\text{ }\text{ }7\text{ }\text{ }7\text{ }\text{ }6\\
&5\text{ }\text{ }5\text{ }\text{ }5\text{ }\text{ }5\text{ }\text{ }7\text{ }\text{ }7\text{ }\text{ }7\text{ }\text{ }7\text{ }\text{ }7\text{ }\text{ }7\text{ }\text{ }7\text{ }\text{ }7\text{ }\text{ }6\text{ }\text{ }6\\
&5\text{ }\text{ }5\text{ }\text{ }5\text{ }\text{ }8\text{ }\text{ }8\text{ }\text{ }0\text{ }\text{ }0\text{ }\text{ }7\text{ }\text{ }7\text{ }\text{ }7\text{ }\text{ }7\text{ }\text{ }7\text{ }\text{ }7\text{ }\text{ }6\text{ }\text{ }6\\
&5\text{ }\text{ }5\text{ }\text{ }8\text{ }\text{ }8\text{ }\text{ }8\text{ }\text{ }0\text{ }\text{ }0\text{ }\text{ }0\text{ }\text{ }0\text{ }\text{ }0\text{ }\text{ }7\text{ }\text{ }7\text{ }\text{ }7\text{ }\text{ }6\text{ }\text{ }6\text{ }\text{ }6\\
&5\text{ }\text{ }8\text{ }\text{ }8\text{ }\text{ }8\text{ }\text{ }8\text{ }\text{ }8\text{ }\text{ }0\text{ }\text{ }0\text{ }\text{ }0\text{ }\text{ }0\text{ }\text{ }0\text{ }\text{ }0\text{ }\text{ }0\text{ }\text{ }7\text{ }\text{ }6\text{ }\text{ }6\text{ }\text{ }6\\
&8\text{ }\text{ }8\text{ }\text{ }8\text{ }\text{ }8\text{ }\text{ }8\text{ }\text{ }8\text{ }\text{ }8\text{ }\text{ }0\text{ }\text{ }0\text{ }\text{ }0\text{ }\text{ }0\text{ }\text{ }0\text{ }\text{ }0\text{ }\text{ }9\text{ }\text{ }6\text{ }\text{ }6\text{ }\text{ }6\text{ }\text{ }6\\
&1\text{ }\text{ }8\text{ }\text{ }8\text{ }\text{ }8\text{ }\text{ }8\text{ }\text{ }8\text{ }\text{ }8\text{ }\text{ }0\text{ }\text{ }0\text{ }\text{ }0\text{ }\text{ }0\text{ }\text{ }0\text{ }\text{ }9\text{ }\text{ }9\text{ }\text{ }9\text{ }\text{ }6\text{ }\text{ }6\text{ }\text{ }6\text{ }\text{ }6\\
&1\text{ }\text{ }1\text{ }\text{ }4\text{ }\text{ }8\text{ }\text{ }8\text{ }\text{ }8\text{ }\text{ }8\text{ }\text{ }8\text{ }\text{ }0\text{ }\text{ }0\text{ }\text{ }0\text{ }\text{ }0\text{ }\text{ }9\text{ }\text{ }9\text{ }\text{ }9\text{ }\text{ }6\text{ }\text{ }6\text{ }\text{ }6\text{ }\text{ }6\text{ }\text{ }3\\
&1\text{ }\text{ }1\text{ }\text{ }1\text{ }\text{ }4\text{ }\text{ }4\text{ }\text{ }8\text{ }\text{ }8\text{ }\text{ }8\text{ }\text{ }8\text{ }\text{ }0\text{ }\text{ }0\text{ }\text{ }0\text{ }\text{ }9\text{ }\text{ }9\text{ }\text{ }9\text{ }\text{ }9\text{ }\text{ }6\text{ }\text{ }6\text{ }\text{ }6\text{ }\text{ }3\text{ }\text{ }3\\
&1\text{ }\text{ }1\text{ }\text{ }1\text{ }\text{ }1\text{ }\text{ }4\text{ }\text{ }4\text{ }\text{ }4\text{ }\text{ }8\text{ }\text{ }8\text{ }\text{ }8\text{ }\text{ }0\text{ }\text{ }0\text{ }\text{ }9\text{ }\text{ }9\text{ }\text{ }9\text{ }\text{ }9\text{ }\text{ }6\text{ }\text{ }6\text{ }\text{ }6\text{ }\text{ }3\text{ }\text{ }3\text{ }\text{ }3\\
&1\text{ }\text{ }1\text{ }\text{ }1\text{ }\text{ }1\text{ }\text{ }1\text{ }\text{ }4\text{ }\text{ }4\text{ }\text{ }4\text{ }\text{ }4\text{ }\text{ }8\text{ }\text{ }0\text{ }\text{ }9\text{ }\text{ }0\text{ }\text{ }9\text{ }\text{ }9\text{ }\text{ }9\text{ }\text{ }9\text{ }\text{ }6\text{ }\text{ }6\text{ }\text{ }3\text{ }\text{ }3\text{ }\text{ }3\text{ }\text{ }3\\
&1\text{ }\text{ }1\text{ }\text{ }1\text{ }\text{ }1\text{ }\text{ }1\text{ }\text{ }1\text{ }\text{ }4\text{ }\text{ }4\text{ }\text{ }4\text{ }\text{ }4\text{ }\text{ }4\text{ }\text{ }9\text{ }\text{ }9\text{ }\text{ }9\text{ }\text{ }9\text{ }\text{ }9\text{ }\text{ }9\text{ }\text{ }6\text{ }\text{ }6\text{ }\text{ }3\text{ }\text{ }3\text{ }\text{ }3\text{ }\text{ }3\text{ }\text{ }3\\
&1\text{ }\text{ }1\text{ }\text{ }1\text{ }\text{ }1\text{ }\text{ }1\text{ }\text{ }1\text{ }\text{ }1\text{ }\text{ }4\text{ }\text{ }4\text{ }\text{ }4\text{ }\text{ }4\text{ }\text{ }4\text{ }\text{ }4\text{ }\text{ }9\text{ }\text{ }9\text{ }\text{ }9\text{ }\text{ }9\text{ }\text{ }9\text{ }\text{ }6\text{ }\text{ }3\text{ }\text{ }3\text{ }\text{ }3\text{ }\text{ }3\text{ }\text{ }3\text{ }\text{ }3\\
&1\text{ }\text{ }1\text{ }\text{ }1\text{ }\text{ }1\text{ }\text{ }1\text{ }\text{ }1\text{ }\text{ }1\text{ }\text{ }1\text{ }\text{ }4\text{ }\text{ }4\text{ }\text{ }4\text{ }\text{ }4\text{ }\text{ }4\text{ }\text{ }4\text{ }\text{ }4\text{ }\text{ }9\text{ }\text{ }9\text{ }\text{ }9\text{ }\text{ }6\text{ }\text{ }3\text{ }\text{ }3\text{ }\text{ }3\text{ }\text{ }3\text{ }\text{ }3\text{ }\text{ }3\text{ }\text{ }3\\
&1\text{ }\text{ }1\text{ }\text{ }1\text{ }\text{ }1\text{ }\text{
}1\text{ }\text{ }1\text{ }\text{ }1\text{ }\text{ }1\text{ }\text{
}1\text{ }\text{ }4\text{ }\text{ }4\text{ }\text{ }4\text{ }\text{
}4\text{ }\text{ }4\text{ }\text{ }4\text{ }\text{ }4\text{ }\text{
}4\text{ }\text{ }9\text{ }\text{ }9\text{ }\text{ }3\text{ }\text{
}3\text{ }\text{ }3\text{ }\text{ }3\text{ }\text{ }3\text{ }\text{
}3\text{ }\text{ }3\text{ }\text{ }3
\endaligned$$



\begin{thebibliography}{99}


\bibitem{BM:top}
Cristiano Bocci, Rick Miranda, \emph{Topics on interpolation
problems in algebraic geometry}, Rend. Sem. Mat. Univ. Politec.
Torino \textbf{62} (2004), 279--334.

\bibitem{CM:degen}
Ciro Ciliberto, Rick Miranda, \emph{Degenerations of planar linear
systems}, J. Reine Angew. Math. \textbf{501} (1998), 191--220.

\bibitem{CM:lin}
Ciro Ciliberto, Rick Miranda, \emph{Linear systems of plane curves
with base points of equal multiplicity},  Trans. Amer. Math. Soc.
\textbf{352} (2000), no. 9, 4037--4050.

\bibitem{Du:red}
Marcin Dumnicki, \emph{Reduction method for linear systems of plane
curves with base fat points}, math.AG/0606716.

\bibitem{DJ:new}
Marcin Dumnicki, Witold Jarnicki, \emph{New effective bounds on the
dimension of a linear system in $\mathbb{P}^2$}, math.AG/0505183.

\bibitem{H:une}
Andr$\acute{e}$ Hirschowitz, \emph{Une conjecture pour la
cohomologie des diviseurs sur les surfaces rationnelles generiques},
J. Reine Angew. Math. \textbf{397} (1989), 208.213.

\bibitem{LJR}
George G. Lorentz, Kurt Jetter, Sherman D. Riemenschneider, Birkhoff
interpolation. Encyclopedia of Mathematics and its Applications, 19.
\emph{Addison-Wesley Publishing Co., Reading, Mass.}, 1983.

\bibitem{LL:biv}
George G. Lorentz,  Rudolph A. Lorentz, \emph{Bivariate Hermite
interpolation and applications to algebraic geometry}, Numer. Math.
\textbf{57} (1990), 669--680.

\bibitem{Na:hil}
Masayoshi Nagata, \emph{On the $14$-th problem of Hilbert}, Amer. J.
Math. \textbf{81} (1959), 766--772.


\bibitem{Ha:nagata}
Brian Harbourne, \emph{On Nagata's conjecture}, J. Algebra
\textbf{236} (2001), no. 2, 692--702.


\bibitem{HR:nagata}
Brian Harbourne, Joaquim  Roe, \emph{Linear systems with multiple
base points in $\Bbb P\sp 2$}, Adv. Geom. \textbf{4} (2004), no. 1,
41--59.


\bibitem{SS:nagata}
Beata Strycharz-Szemberg, Tomasz Szemberg, \emph{Remarks on the
Nagata conjecture}, Serdica Math. J. \textbf{30} (2004), no. 2-3,
405--430.


\end{thebibliography}
\end{document}